\documentclass[12pt]{amsart}

\usepackage{amsmath,amssymb,amsthm,enumerate}

\newtheorem{theorem}{Theorem}[section]

\newtheorem{proposition}[theorem]{Proposition}

\theoremstyle{definition}
\newtheorem{definition}[theorem]{Definition}

\newtheorem{question}[theorem]{Question}

\theoremstyle{remark}
\newtheorem{remark}[theorem]{Remark}

\newcommand{\bD}{\mathbb D}
\newcommand{\bP}{\mathbb P}
\newcommand{\bC}{\mathbb C}
\newcommand{\bR}{\mathbb R}

\newcommand{\bT}{\mathbb T}
\newcommand{\bH}{\mathbb H}
\newcommand{\mU}{\mathcal U}

\def\AC{\mathcal A\mathcal C}

\newcommand{\beq}{\begin{equation}}
\newcommand{\eeq}{\end{equation}}

\newcommand{\eps}{\epsilon}

\newcommand{\norm}{|\!|}

\DeclareMathOperator{\im}{Im}
\DeclareMathOperator{\re}{Re}

\begin{document}
\title[]{The umbilical locus on the boundary of strictly pseudoconvex domains in $\bC^2$.}
\author{Peter Ebenfelt}
\address{Department of Mathematics, University of California at San Diego, La Jolla, CA 92093-0112}
\email{pebenfel@math.ucsd.edu}

\date{\today}
\thanks{The author was supported in part by the NSF grant DMS-1600701.}

\begin{abstract} The main objective of this paper is to survey some recent results on the Chern--Moser question concerning existence of umbilical points on three dimensional CR submanifolds in $\mathbb C^2$. \end{abstract}


\maketitle

\section{Introduction}

In their seminal paper \cite{CM74} from 1974, S.-S. Chern and J. K. Moser posed the problem of understanding what compact strictly pseudoconvex three-dimensional CR manifolds lack (CR) umbilical points. By the work of E. Cartan \cite{Cartan32,Cartan33}, in which he classified homogeneous three-dimensional strictly pseudoconvex CR manifolds, it is known that there are homogeneous, compact and strictly pseudoconvex CR manifolds of dimension three without umbilical points. In fact, Cartan showed that the only homogeneous such manifolds are
\begin{equation}\label{mualph}
\mu_\alpha:=\{[z_0:z_1:z_2]\in \bC\bP^2\colon |z_0|^2+|z_1|^2+|z_2|^2=\alpha
|z_0^2+z_1^2+z_2^2|\},\quad \alpha >1,
\end{equation}
and their covers, later classified in \cite{Isaev06} as a $4\!:\! 1$ cover $\mu_\alpha^{(4)}$ (diffeomorphic to a sphere) that factors through a $2\!:\! 1$ cover $\mu_\alpha^{(2)}$ (consisting of the intersection of sphere and a nonsingular holomorphic quadric in $\bC^3$). It is known that none of these embed in $\bC^2$, and that those that are diffeomorphic to a sphere, $\mu_\alpha^{(4)}$, do not embed in $\bC^n$ for any $n$. These observations led to the following long-standing and well known question concerning the geometry of domains and their boundaries in several complex variables:
{\it Does there exist a bounded strictly pseudoconvex domain $D\subset\bC^2$ with smooth boundary $M:=\partial D$ such that $M$ has no (CR) umbilical points?} This question was recently settled by the author, jointly with N. S. Duong and D. Zaitsev, in \cite{EDZ16} where the following theorem was proved:

\begin{theorem}[\cite{EDZ16}]\label{T:main0}
For any $\eps>0$, let $D_\eps$ be the bounded strictly pseudoconvex domain in $\bC^2$ given by
\beq\label{Deps}
(\log|z|)^2+(\log|w|)^2<\eps^2.
\eeq
 The boundary $M_\eps:=\partial D_\eps\subset \bC^2$ is a compact strictly pseudoconvex CR manifold without umbilical points.
\end{theorem}
The boundaries $M_\eps$ are diffeomorphic to the 3-torus $\mathbb T^3$, and in particular have nontrivial fundamental group. In Section \ref{S:Grauert} below, we explain where the domains in Theorem \ref{T:main0} come from, and show (Theorem \ref{T:main1}) that there are smoothly bounded strictly pseudoconvex domains in $\bC^2$, without umbilical points on the boundary and with arbitrarily large negative Euler characteristics. (The domains $D_\eps$ in Theorem \ref{T:main0} have Euler characteristic 0.) It is still unknown whether there are examples that are simply connected. A refined Chern--Moser question can then be formulated as follows:

\begin{question}\label{Q1} {\it Does there exist a bounded strictly pseudoconvex domain $D\subset\bC^2$ with smooth boundary $M:=\partial D$ such that $M$ has no (CR) umbilical points {\bf and} $M$ is diffeomorphic to the sphere?}
\end{question}

\begin{remark}\label{Q1rem} Recall that the only {\it homogeneous} compact strictly pseudoconvex CR manifolds of dimension three that have no umbilical points and are also diffeomorphic to the sphere are the 4:1 covers $\mu_\alpha^{(4)}$ of E. Cartan's $\mu_\alpha$. The CR manifolds $\mu_\alpha^{(4)}$ coincide with H. Rossi's examples of non-embeddable CR manifolds in \cite{Rossi65} and are therefore not embeddable in $\bC^n$ for any $n$. The following weaker version of Question \ref{Q1} is also open: {\it Does there exist a compact strictly pseudoconvex three-dimensional CR manifold that has no umbilical points, is diffeomorphic to the sphere, and is embeddable in $\bC^n$ for some $n$?}
\end{remark}

The purpose of this paper is to survey what is known (to the author) about the questions posed above (regarding existence of umbilical points), and also to point out some additional questions and problems on the topic that are still open. The main results presented here are based on the author's joint work with S.N. Duong and D. Zaitsev in \cite{EDumb15}, \cite{EZ16}, and \cite{EDZ16}.

Before proceeding, however, we should mention that these questions are particular to $\bC^2$ (or, more precisely, to CR manifolds of dimension three). It became clear already in \cite{CM74} that there is sharp difference between umbilical points on boundaries of strictly pseudoconvex domains in $\bC^2$ and in $\bC^n$ with $n\geq 3$. To begin with, umbilical points in $\bC^n$ with $n\geq 3$ are determined by the vanishing of a 4th order tensor (the CR curvature), whereas umbilical points in $\bC^2$ are determined by the vanishing of a 6th order tensor (discovered already by E. Cartan \cite{Cartan32,Cartan33} in the early 1930's, and usually referred to as "E. Cartan's 6th order tensor"). A simple Thom transversality argument (see, e.g., \cite{EZ16}) shows that a generic (i.e., sufficiently general) strictly pseudoconvex domain in $\bC^n$ with $n\geq 4$ does not have any umbilical points in its boundary, and Webster \cite{Webster00} showed that, in particular, every non-spherical real ellipsoid in $\bC^n$ with $n\geq 3$ has no umbilical points. On the other hand, in contrast with Webster's result and showcasing the point that the situation in $\bC^2$ and that in $\bC^n$ with $n\geq 3$ is different, X. Huang and S. Ji \cite{HuangJi07} proved that every real ellipsoid in $\bC^2$ must have umbilical points. An application of Thom transversality in the case of strictly pseudoconvex domains $\Omega$ in $\bC^2$ shows that generically the set of umbilical points on the boundary $M=\partial\Omega$ is either empty (as in Theorem \ref{T:main0}) or form smooth curves in $M$ (as in Theorem \ref{T:mainC} below).

The paper is organized as follows. In Section \ref{SS:neg}, we first present a negative answer to Question \ref{Q1} for domains with a transverse and free circle action. In Section \ref{S:umbpts}, we present some preliminary material on umbilical points, including their definition as the zero locus of E. Cartan's 6th order tensor as well as a global formula for this tensor. Section \ref{S:Grauert} introduces Grauert tubes and their boundaries, which provide a context for Theorem \ref{T:main0}. In particular, Theorem \ref{T:main1} in that section yields additional domains in $\bC^2$ without umbilical points on the boundary and with arbitrarily large negative Euler characteristics. In Section \ref{S:pert}, we discuss the existence of umbilical points on perturbations of the sphere. The main result in this section is also a negative answer to Question \ref{Q1} in the context of suitably generic almost circular perturbations.

\subsection{A negative result}\label{SS:neg} In joint work with S. N. Duong \cite{EDumb15}, the author has shown that the answer to Question \ref{Q1} is negative provided that the domain is assumed to have additional symmetry. Recall that a domain $D\subset \bC^n$ is said to {\it complete circular} if $Z\in D$ implies that the disk $\{\zeta Z\colon \zeta\in \bD\}$ is contained in $D$; here, $\bD$ denotes the unit disk in $\bC$. The following result was proved in \cite{EDumb15}:

\begin{theorem}[\cite{EDumb15}]\label{T:mainC} Let $D\subset \bC^2$ be a bounded, strictly pseudoconvex, complete circular domain with smooth boundary $M:=\partial D$. Then, $M$ has a non-empty locus of umbilical points.
\end{theorem}

Theorem \ref{T:mainC} is a consequence of a more general result concerning the existence of umbilical points on CR manifolds with a circle action. Let $M$ be a compact, strictly pseudoconvex, three dimensional CR manifold and assume that there is a free action of $U(1)$ on $M$ via CR automorphisms, such that the action is everywhere transverse to the CR tangent spaces of $M$. We shall let $X$ denote the smooth compact surface obtained by $\pi\colon M\to X:=M/U(1)$. This Riemann surface can be given a complex structure by $T^{1,0}X:=\pi_* T^{1,0}M$ and $M$ can in fact be identified with the unit circle bundle in a positive (or negative) holomorphic line bundle over $X$; see \cite{Epstein92} for details. We note that if $p\in M$ is an umbilical point, then the entire $U(1)$-orbit $\pi^{-1}(\pi(p))$ is umbilical. In \cite{EDumb15}, the following result, which is more general than Theorem \ref{T:mainC}, was also proved.

\begin{theorem}[\cite{EDumb15}]\label{T:mainS1} Let $M$ be a compact, strictly pseudoconvex, three dimensional CR manifold with a transverse, free {\rm CR} $U(1)$-action. If the compact surface $X:=M/U(1)$ is not a torus (i.e., does not have Euler characteristic zero), then the locus of umbilical points contains at least one $U(1)$-orbit.
\end{theorem}

\begin{remark}
We note that the domains in Theorem \ref{T:main0}, being Reinhardt domains, have a free {\rm CR} $U(1)$-action. However, this action is not {\it transverse}, and Theorem \ref{T:mainS1} does not apply.
\end{remark}

Theorem \ref{T:mainS1} in turn follows from a Poincar\'e-Hopf type index formula for umbilical circles in the setting of three dimensional CR manifolds $M$ with a transverse, free {\rm CR} $U(1)$-action. The local index $\iota_p$ of the umbilical locus $\mU\subset M$ at a point $p\in U$ where $\mU$ is a smooth curve was introduced in \cite{EZ16} and is described in Section \ref{SS:umbind} below. In the context of strictly pseudoconvex, three dimensional CR manifolds with a transverse, free {\rm CR} $U(1)$-action, the umbilical locus $\mU$ consists of circles $\pi^{-1}(\zeta)$ for $\zeta \in X=M/U(1)$. The index of an isolated umbilical circle $O_\zeta:=\pi^{-1}(\zeta)$ was introduced in \cite{EDumb15} as the local index $\iota_p$ for any $p\in \pi^{-1}(\zeta)$, and the following index formula was proved for $M$ such that $\mU$ consists of isolated circles,
\beq\label{PHindex}
\sum_{\zeta\in X} \iota(O_\zeta)=\chi(X),
\eeq
where $\chi(X)=2-2g$ denotes the Euler characteristic of $X=M/U(1)$. Theorem \ref{T:mainS1} follows immediately from this formula.

\section{Umbilical points on three-dimensional CR manifolds}\label{S:umbpts}

In the theory of strictly pseudoconvex CR manifolds, the role of a flat model is played by the Euclidian sphere in $\bC^{n+1}$. Roughly speaking, and mimicking the geometric meaning of umbilical points in classical geometry of surfaces in $\bR^3$, we say that a point $p$ on a strictly pseudoconvex hypersurface $M=M^{2n+1}$ (where the superscript $2n+1$ denotes the real dimension of $M$) is {\it umbilical} if there is a formal holomorphic embedding $Z$ of $M$ at $p$ into $\bC^{n+1}$ such that $M$ can be approximated at $0=Z(p)$ to a higher than expected order by a sphere through $0\in Z(M)$. By the work of Chern--Moser \cite{CM74}, the expected order is 3 when $n\geq 2$, and 5 when $n=1$. As indicated in the introduction, we shall focus on the case $n=1$ in this paper. In this case, there is a formal embedding $Z=(z,w)\in \bC^2$ of $M=M^3$ at $p$ such that the formal image, still denoted $M$ here by a slight abuse of notation, is given as a formal graph of the form
\beq\label{CMnormal}
\im w=|z|^2+\sum_{k,l\geq 2}A_{kl}(z,\bar z,\re w), \quad A_{kl}=\overline{A_{lk}},
\eeq
where each $A_{kl}$ is a polynomial of bidegree $(k,l)$ in $(z,\bar z)$ with coefficients that are power series in $s$. Moreover, a normal form, which is unique up to the action of the stability group at $0$ of the Heisenberg representation of the sphere (given by \eqref{CMnormal} with all $A_{kl}=0$), can be achieved with
\beq\label{CMnormalcond}
A_{2,2}=A_{2,3}=A_{3,3}=0.
\eeq
Such a formal coordinate system $Z=(z,w)$ is called a Chern--Moser normal coordinate system and the defining equation \eqref{CMnormal} is then said to be in Chern--Moser normal form. The lowest order terms that can appear in \eqref{CMnormal} are the terms of degree 6 given by \beq\label{c24}
A_{2,4}(z,\bar z,0)=2\re(c_{2,4}z^2\bar z^4).
\eeq
While the coefficients in the Chern-Moser normal form \eqref{CMnormal} are not invariants in general (as the normal form is not unique), the property of the coefficient $c_{2,4}$ being zero or non-zero is easily seen to be invariant. It is thus clear that the point $p\in M$ is umbilical precisely when the Chern-Moser coefficient $c_{2,4}=0$. Moreover, we may in fact consider the coefficient $c_{2,4}$ as an invariant (usually called E. Cartan's 6th order tensor or the umbilical tensor) by noting the following transformation rule under changes of Chern--Moser normal form.

\begin{proposition}\label{P:CartanTensor} Let $Z^*=(z^*,w^*)$, $Z=(z,w)$ be Chern--Moser normal coordinates for $M$ at $p$, and let $Z^*=H(Z)$ be the corresponding formal biholomorphic transformation. Then the following transformation rule holds for the coefficient $c_{2,4}$:
\beq\label{CTrule}
c_{2,4}=(\det H_Z(0))^{-1/3}\overline{(\det H_Z(0))}\,{}^{5/3}c_{2,4}^*.
\eeq
\end{proposition}

\begin{proof} To convince oneself that if there is transformation rule of the form
\beq\label{CTrule-2}
c^*_{2,4}=(\det H_Z(0))^{a}\overline{(\det H_Z(0))}\,{}^{b}c_{2,4},
\eeq
for some $a,b$, then it must be given by \eqref{CTrule} (i.e., $a=-1/3$, $b=5/3$), one should consider first the simple biholomorphic mappings
\beq\label{simpletrans}
(z^*,w^*)=(\delta z,\delta^2 w),\quad (z^*,w^*)=(e^{it}z,w),\quad \delta>0,\ t\in \bR,
\eeq
which preserve Chern--Moser normal form without further normalization. It follows readily that the coefficient $c_{2,4}$ satisfies the transformation rule \eqref{CTrule} under these mappings and there are no other $a,b$ such that \eqref{CTrule-2} can hold. To prove the statement in general, one may follow the calculations and arguments in \cite{Graham87} (see, e.g., Lemma 2.8).
\end{proof}

We remark that the definition of the umbilical tensor in terms of the Chern--Moser normal forms is not particularly convenient to use in the study of the locus of umbilical points, as this definition requires a renormalization process at each point to check whether the point is umbilical. We shall instead describe a more convenient formula, first introduced in \cite{EZ16}.

\subsection{A global formula for the umbilical tensor} Let $M=M^3$ be strictly pseudoconvex hypersurface in $\bC^2$ defined by a real equation $\rho=0$ with $\rho\in C^\infty$ and $d\rho|_M\neq 0$. We let $L$ be the $(1,0)$-vector field
\beq\label{L}
L=-\rho_w\frac{\partial}{\partial z}+\rho_z\frac{\partial}{\partial w},
\eeq
which is tangent to $M$. In \cite{EZ16}, the following matrix was introduced,
\beq\label{a3}
A_3=A_3(\rho):=
\begin{pmatrix}
\rho_w^3  & \bar L(\rho_w^3) & \cdots & \bar L^4(\rho_w^3)\cr
 \rho_z\rho_w^2 & \bar L( \rho_z\rho_w^2) & \cdots & \bar L^4( \rho_z\rho_w^2)\cr
  \rho_z^2\rho_w & \bar L(  \rho_z^2\rho_w) & \cdots & \bar L^4(  \rho_z^2\rho_w)\cr
\rho_z^3 & \bar L(\rho_z^3) & \cdots & \bar L^4(\rho_z^3)\cr
\rho_{Z^2}(L, L) & \bar L(\rho_{Z^2}(L, L)) & \cdots & \bar L^4(\rho_{Z^2}(L, L))
\end{pmatrix},
\eeq
and it was shown that its determinant has certain invariance properties under rescalings and changes of coordinates. Let $\rho^*=a\rho$, with $a\in C^\infty$ and $a\neq 0$, and let $Z^*=H(Z)$ be biholomorphic mapping. Then, with $A_3=A_3(\rho)$ in the coordinates $Z$ and $A_3^*=A_3(\rho^*)$ in the coordinates $Z^*$, we have on $M$
\beq\label{a3invariance}
a^{25}\det A_3=(\det H_Z)^8\overline{(\det H_Z)}\,^{10}\det A^*_3.
\eeq
It was also shown in \cite{EZ16} that in Chern--Moser normal coordinates $Z=(z,w)$ and with $\rho$ in Chern--Moser normal form,
\beq\label{a3=c}
\det A_3|_{Z=0}=B\, c_{2,4},
\eeq
for some universal non-zero constant $B$. Thus, one concludes from \eqref{a3invariance} and \eqref{a3=c} that $\det A_3$ equals a non-zero function times the umbilical tensor, and one may study the locus of umbilical points by considering the zero locus of $\det A_3$.

It is in fact possible to use $\det A_3$ to obtain an exact representation of the umbilical tensor itself in any coordinate system $Z=(z,w)$ and using any defining function $\rho$. To this end, we introduce Fefferman's complex Monge--Ampere operator,
\beq\label{FeffJ}
J=J(\rho):=- \det
\begin{pmatrix}
\rho & \rho_{\bar z} & \rho_{\bar w}\\
\rho_{z} &\rho_{z\bar z} &\rho_{z\bar w}\\
\rho_w & \rho_{w\bar z} & \rho_{w\bar w}
\end{pmatrix},
\eeq
which is easily verified to satisfy
\beq\label{FeffJ-2}
J=\tilde J+O(\rho),
\eeq
where
\beq\label{FeffJ-2}
\tilde J:=
-\det \begin{pmatrix}
\rho_z & \bar L\rho_z\\
\rho_w & \bar L\rho_w
\end{pmatrix}.
\eeq
Recall that the determinant $J|_M=\tilde J|_M$ vanishes precisely at the points where the Levi form of $M$ degenerates. Thus, if $M$ is a strictly pseudoconvex hypersurface in $\bC^2$, then we may introduce an invariant $Q$, defined near $M$, as follows:
\beq\label{Q}
Q:=\frac{\det A_3}{\tilde J^{25/3}}.
\eeq
In Chern--Moser normal cordinates $Z$ and with $\rho$ in Chern--Moser normal form, the determinant $\tilde J$ equals a universal constant and, therefore, we have
\beq\label{Q=c}
Q|_{Z=0}=B'\, c_{2,4},
\eeq
for some other universal constant $B'\neq 0$. It follows from \eqref{Q}, Proposition \ref{P:CartanTensor}, and the following proposition that $Q$ so defined is independent of the choice of defining function and represents the umbilical tensor in any coordinate system:

\begin{proposition}\label{P:Qinv}
Let $Z^*$ and $Z$ be any coordinate systems for $M$ near $p\in M$, and let $Z^*=H(Z)$ be the corresponding formal biholomorphic transformation. Let $Q$ and $Q^*$ denote the invariants given by \eqref{Q} using defining equations $\rho$ and $\rho^*=a\rho$, for some non-vanishing real function $a$, in the coordinate systems $Z$ and $Z^*$, respectively. Then the following transformation rule holds:
\beq\label{QTrule}
Q=(\det H_Z)^{-1/3}\overline{(\det H_Z)}\,{}^{5/3}Q^*.
\eeq
\end{proposition}

\begin{proof} This follows immediately from \eqref{a3invariance} and the corresponding transformation rule (well known and also readily verified by calculations similar to those in \cite{EZ16}) for $\tilde J$:
\beq\label{JTrule}
a^3\tilde J=(\det H_Z)\overline{(\det H_Z)}\tilde J^*.
\eeq
\end{proof}

We note that Proposition \ref{P:Qinv} establishes $Q$ as a section of the line bundle
$$
{K_M}^{1/3}\otimes \overline{K_M}\,^{-5/3}\to M,
$$
where $K_M$ denotes canonical line bundle of $M$, or equivalently the restriction of the canonical line bundle $K^*_{\bC^2}\to \bC^2$ to $M$, and $\overline{K}$ denotes the conjugate of a complex line bundle $K$. The umbilical locus $\mU$ of $M$ is now defined as the zero locus of $Q$. Since $M\subset \bC^2$, we may fix a coordinate system $Z$ and thus we may, and we shall, identify $Q$ with a function on $M$, given by \eqref{Q}. For practical purposes, we note that $\mU$ can also be defined by the vanishing of the function $\det A_3=\det A_3(\rho)$ for any choice of defining function $\rho$ for $M$.

\subsection{The umbilical index of curves}\label{SS:umbind} If $\Gamma$ is a closed oriented curve on $M$ that does not intersect the umbilical locus $\mU$, then its umbilical index $I_\Gamma$ is defined to be -1/2 times the winding number of $Q$ around $\Gamma$, and can be computed via the integral
\beq\label{uiGamma}
I_\Gamma:=-\frac{1}{4\pi}\int_\Gamma\frac{dQ}{Q}.
\eeq
If $\Sigma$ is an oriented 2-surface in $M$ that intersects $\mU$ at isolated points, then the umbilical intersection index at a point $p\in \Sigma\cap\mU$, denoted $\iota_{\Sigma_p}$, is defined to be the umbilical index of $\Gamma:=\partial\Sigma_p$, where $\Sigma_p$ is a sufficiently small disk in $\Sigma$ through $p$. Thus, by Stokes Theorem, we have
\beq\label{stokesindex}
I_{\partial\Sigma}=\sum_{p\in \Sigma\cap\mU}\iota_{\Sigma_p}.
\eeq
When $\mU$ is a smooth curve at $p\in\mU$, then the local index at $p$, denoted $\iota_p$, is defined as the umbilical intersection index at $p$ with any 2-surface that intersects $\mU$ transversally at $p$. We note that we may use $\tilde Q:=\det A_3=\det A_3(\rho)$, for any choice of defining function $\rho$, in place of $Q$ to compute indices, since as noted above $\tilde Q=e Q$ for some non-vanishing function $e$ and, hence, $d\tilde Q/\tilde Q=dQ/Q$ modulo an exact form.

We observe that if $M$ has no umbilical points ($Q\neq 0$ on $M$), then $dQ/Q$ is exact and the index of any curve is zero. Thus, the existence of a closed curve with non-zero index guarantees the existence of a non-trivial umbilical locus. We also observe that if the umbilical locus of $M$ possesses a component $\mU_1$ such that $\mU_1$ is a curve whose smooth points have non-zero local index, then small perturbations of $M$ will also have an umbilical locus with the same property. We shall refer to such umbilical points with non-zero local index as {\it stable} umbilical points.

\section{Grauert tubes and non-umbilical CR manifolds}\label{S:Grauert}

Let $(X=X^n,g)$ be a real-analytic Riemannian manifold, and let $\tilde X=\tilde X^n$ be a complexification of $X$. The complexification $\tilde X$ is a complex manifold of (complex) dimension $n$ containing $X$ as a totally real $n$-dimensional manifold. By the work of Lempert-Sz\H oke \cite{LempertSzoke91}, Guillemin--Stenzel \cite{GuilleminStenzel91}, there is a uniquely determined K\"ahler potential $\rho$ in a neighborhood of $X$ in $\tilde X$ satisfying the following conditions:
\begin{itemize}
\item[(i)] $X$ is given by $\{\rho=0\}$;
\item[(ii)] $\tilde g|_X=g$, where $g$ is the given Riemannian metric on $X$ and $\tilde g$ is the K\"ahler metric corresponding to the K\"ahler form $\omega:=i\partial\bar\partial\rho/2$;
\item[(iii)] $(\partial\bar\partial\sqrt{\rho})^n=0$ on $\tilde X\setminus X$.
\end{itemize}
When $X$ is compact, there is $r_0>0$ such that the complex manifold $\tilde X_r:=\{\rho<r^2\}$, for $0<r<r_0$, is relatively compact in $\tilde X$. The manifold $\tilde X_r$ is called the Grauert tube of $(X,g)$ of radius $r$.

The Grauert tubes can also be realized as the disk bundles $T^rX:=\{(x,v)\in TX\colon \norm v\norm^2<r^2\}$ by defining a unique complex structure on the tangent bundle $TX$, near the zero section, in which $\rho(x,v)=2\norm v\norm^2$ satisfies the conditions (i)--(iii) above.

\subsection{Non-umbilical CR manifolds as Grauert tubes} It was shown by Patrizio-Wong \cite{PatrizioWong91} that Cartan's families of CR manifolds $\mu_\alpha\subset\bC\bP^2$ and $\mu^{(2)}_\alpha\subset\bC^3$, for $\alpha>1$, realize boundaries of the Grauert tubes $\tilde X_r$ and $\tilde X_r^{(2)}$ of the real projective 2-space $X=\bR\bP^2$ with its constant curvature metric and the 2-sphere $X=S^2$ with the round (constant curvature) metric, respectively. These Grauert tubes cannot be embedded into $\bC^2$, because if they could then their centers $S^2$ and $\bR\bP^2$ would be embedded as totally real submanifolds in $\bC^2$, which is impossible by a well known result of Wells (\cite{Wells69}; see also \cite{Bishop65} and \cite{Forstneric92}). The only Grauert tubes that can be embedded in $\bC^2$, for the same reason, are those where the center is either the 2-torus $\bT^2$ (the compact orientable case with genus 1) or a compact non-orientable 2-surface (Klein surface) of genus 2 mod 4 (see \cite{Forstneric92}), and these can indeed be embedded in $\bC^2$, at least for sufficiently small radii $r$. The key to Theorem \ref{T:main0} is that the strictly pseudoconvex domains $D_\eps$ are actually Grauert tubes of the 2-torus $X=\bT^2$ equipped with the flat metric.

\begin{proof}[Sketch of Proof of Theorem $\ref{T:main0}$] It is not difficult to check that the Grauert tube of radius $r$ of $\bR^2$ with the (Euclidian) flat metric is given by
$T^r=T^r\bR^2:=\{(\zeta,\xi)\in \bC^2\colon \rho<r^2\}$, where
\beq\label{noncompGr}
\rho=\rho(\zeta,\xi,\bar \zeta,\bar \xi):=\left(\im \zeta\right)^2+\left(\im \xi\right)^2.
\eeq
The boundary $\partial T^r$ is one of E. Cartan's homogeneous, noncompact, non-spherical, strictly pseudoconvex CR manifolds of dimension three \cite{Cartan32,Cartan33}. In particular, since $\partial T^r$ is homogeneous and non-spherical, it is non-umbilical at every point. Next, we let $\Lambda$ be the subgroup of the group of rigid motions (isometries) of $\bR^2$ that generate the lattice spanned by $(2\pi,0)$, $(0,2\pi)$, so that the standard 2-torus $\bT^2$ is given by $\bR^2/\Lambda$. The potential $\rho$ in \eqref{noncompGr} is invariant under $\Lambda$ and therefore it descends to $\bC^2/\Lambda$, where we still denote it by $\rho$. The Grauert tube of $(X=\bT^2, g_{flat})$,  where $g_{flat}$ denotes the flat metric on the 2-torus, of radius $r$ is then given as the domain in $\bC^2/\Lambda$ by $\tilde X_r:=\{\rho<r^2\}$. The domain $D_\eps$ in Theorem \ref{T:main0} is now the image of $\tilde X_r$, with $r=\eps/2$, under the biholomorphic map $H(\zeta,\xi):=(e^{i\zeta}, e^{i\xi})$.
\end{proof}

We observe that the Grauert tubes of the Klein Bottle $K=\bR\bP^2\#\bR\bP^2$ (the compact non-orientable surface of genus 2) can be constructed in a similar way by considering instead the subgroup $\Lambda$ that generate $K=\bR^2/\Lambda$. These Grauert tubes can also be embedded in $\bC^2$, for sufficiently small radii $r$, by complexifying any real-analytic embedding of $K$ (see \cite{Rudin81} for an explicit such embedding).

Grauert tubes of constant curvature Riemann surfaces of higher genus can be constructed by starting with hyperbolic space $\bH^2$ and its constant scalar curvature metric $g_{csc}$. The Grauert tubes of $(\bH^2,g_{csc})$ can be realized as the complex manifolds $\Omega_\alpha\subset \bC^3$, for $-1<\alpha<1$, given by
\beq\label{GrHyp}
\begin{aligned}
z_1^2+z_2^2-z_3^2 &=-1,\\
|z_1|^2+|z_2|^2-|z_3|^2 &<\alpha.
\end{aligned}
\eeq
In this model, the hyperbolic space $\bH^2$ corresponds to the set of real points on $z_1^2+z_2^2-z_3^2=-1$ such that $z_3=x_3>0$. The action of the space of isometries of $\bH^2$ on the Grauert tubes is via the orthogonal group $O_+(2,1)$ preserving $x_3>0$. The K\"ahler potential $\rho$ is an increasing function of $|z_1|^2+|z_2|^2-|z_3|^2\in (-1,1)$; indeed, it was shown in \cite{Kan96} that $\rho$ is given by
\beq\label{rhohyp}
\rho=\left(\arccos(|z_1|^2+|z_2|^2-|z_3|^2)-\pi\right)^2/2,
\eeq
and the Grauert tubes can be defined as $\Omega_\alpha=\{\rho<r^2,\ r=r(\alpha)\}$.
The function $\rho$ is clearly invariant under the action of $O_+(2,1)$. Recall that any compact orientable surface (Riemann surface) of genus $\geq 2$ and any compact nonorientable surface (Klein surface) of genus $>2$ can be represented as $\bH^2/\Lambda$, where $\Lambda$ is a discrete subgroup  (an NEC group; Fuchsian in the orientable case) of the space of isometries of $\bH^2$. Consequently, the Grauert tube of radius $r$, for sufficiently small $r$, of any compact constant curvature surface $X$, of genus $\geq 2$ in the orientable case and $>2$ in the non-orientable case, can be obtained by taking a quotient of $\Omega_\alpha$ by the action of $\Lambda$. Moreover, one can show (see e.g. \cite{Kan96}; also in \cite{Cartan32,Cartan33}) that the boundaries of these Grauert tubes $\tilde X_r=\Omega_\alpha/\Lambda$ are non-spherical, except for one specific radius $r=r_{sph}$. Since the boundaries $\partial\tilde X_r$ are locally CR equivalent to the homogeneous, non-compact boundaries $\partial\Omega_\alpha$, and hence locally homogeneous, we conclude that that the boundary $\partial\tilde X_r$ of the Grauert tube of radius $r\neq r_{sph}$ of such a compact, constant curvature surface $X$ has no umbilical points. Recall that a compact {\it orientable} 2-surface of genus $\geq 2$ cannot be embedded as a totally real submanifold in $\bC^2$, and hence its Grauert tubes cannot be embedded in $\bC^2$. However, any compact, non-orientable surface $X$ of genus 2 mod 4 can be real-analytically embedded as a totally real submanifold in $\bC^2$ (see, e.g., \cite{Forstneric92}) and, hence, its Grauert tubes $\tilde X_r$, for sufficiently small $r>0$, can be holomorphically embedded in $\bC^2$. This provides examples of domains in $\bC^2$, without umbilical points on the boundary, that are Grauert tubes of compact (non-orientable) surfaces of arbitrarily high genus.

\subsection{Domains in $\bC^2$ with non-umbilical boundary and large negative Euler characteristics} The discussion in the preceding subsection allows us to formulate the following theorem:

\begin{theorem}\label{T:main1} For every integer $\chi\leq 0$ such that $\chi=0 \mod 4$, there exists a bounded strictly pseudoconvex domain $D\subset \bC^2$ such that
\begin{itemize}
\item[(i)] The boundary $M:=\partial D\subset \bC^2$ is a compact strictly pseudoconvex CR manifold without umbilical points.
\item[(ii)] $D$ is homotopy equivalent to a compact surface $X$ with Euler characteristic $\chi$. If $\chi<0$, then $X$ is non-orientable.
\end{itemize}
\end{theorem}

\begin{proof} Recall that a compact non-orientable surface $X$ of genus $h\geq 1$ has Euler characteristic $\chi=2-h$. Thus, for such surfaces of genus $h=2\mod 4$, we obtain $\chi=0\mod 4$ and every $\chi\leq 0$ with this property is realized for some $h=2\mod 4$. We also recall that the Grauert tube of radius $r$ of a surface $X$ can be realized as the disk bundle of radius $r$ in the tangent bundle, and hence is homotopy equivalent to $X$. The conclusions of the theorem now follow from the discussion in the previous subsection by taking $D$ to be a Grauert tube of sufficiently small radius $0<r\neq r_{sph}$ of a compact non-orientable surface $X$ of genus $h=2-\chi$ (or if $\chi=0$ one can instead take $X$ to be the 2-torus, which also has Euler characteristic 0), equipped with the constant curvature metric.
\end{proof}

\begin{remark}
We note that the Chern--Moser Question \ref{Q1} essentially asks for a smoothly bounded strictly pseudoconvex domain $D\subset \bC^2$, without umbilical points on its boundary and such that $D$ is diffeomorphic to a ball. Such a domain would then be contractible (homotopic to a point) and would therefore have Euler characteristic one. Another less specific question, motivated by Theorem \ref{T:main1}, would be: {\it Are there smoothly bounded strictly pseudoconvex domains $D\subset \bC^2$, without umbilical points on its boundary and with positive Euler characteristic?}
\end{remark}

\section{Perturbations of the sphere}\label{S:pert}

Motivated by Question \ref{Q1}, we shall consider small perturbations $M_\eps$ of the unit sphere $S^{3}\subset \bC^2$. We consider a defining function $\rho^\eps$, for small $\eps>0$, of the form
\beq\label{rhoeps}
\rho^\eps=\rho^0+\eps\rho'+O(\eps^2),\quad \rho^0=|z|^2+|w|^2-1,
\eeq
where $\rho'$ is a smooth real-valued function of $z$ and $w$, and $\rho^\eps-\rho^0$ has no constant or linear term. We observe that the tangent $(1,0)$-vector field $L^\eps=-\rho^\eps_w\partial/\partial z+\rho^\eps_z\partial/\partial w$ is of the form
\beq\label{Leps}
L_\eps=L_0+\eps L'+O(\eps^2),\quad L_0=-\bar w\frac{\partial}{\partial z} + \bar z\frac{\partial}{\partial w}.
\eeq
The following was observed in \cite{EZ16}:

\begin{proposition}[\cite{EZ16}]\label{Lin-eps} For a perturbation of the form \eqref{rhoeps},
\beq\label{Lin-epsterm}
\det A_3(\rho^\eps)=c_0\bar L_0^4 (\rho'_{Z^2}(L_0,L_0))\eps+O(\eps^2),
\eeq
where $c_0$ is a universal polynomial that does not vanish on the unit sphere $\rho^0=0$.
\end{proposition}


We shall denote by $Q^0$ the linear partial differential operator
\beq\label{Q0}
Q^0(R):=\bar L_0^4 (R_{Z^2}(L_0,L_0)),
\eeq
which appears in the leading term in the asymptotic expansion of $\det A_3(\rho^\eps)$.
Thus, we can detect umbilical points on $M_\eps$, for sufficiently small $\eps>0$, by finding a closed curve $\Gamma$ on the unit sphere $M_0=S^3$ such that
\beq\label{Q'}
Q':=Q^0(\rho')=\bar L_0^4 (\rho'_{Z^2}(L_0,L_0))
\eeq
has non-vanishing winding number around $\Gamma$:

\begin{proposition}\label{P:MepsQ'}
Let $M_\eps$ be defined by $\rho^\eps=0$, where $\rho^\eps$ is given by \eqref{rhoeps}, and let $Q'$ be given by \eqref{Q'}.
If there exists a closed curve $\Gamma$ on $M_0=S^3$ such that $Q'\neq 0$ on $\Gamma$ and
\beq\label{windQ'}
\int_\Gamma\frac{dQ'}{Q'}\neq 0,
\eeq
then, for sufficiently small $\eps>0$, the umbilical locus of $M_\eps$ contains at least one non-trivial curve. More precisely, the umbilical locus contains either a 2-surface, or a curve of stable umbilical points.
\end{proposition}

\begin{proof}
This follows from Proposition \ref{Lin-eps}, the discussion in Section \ref{SS:umbind}, and a simple approximation argument left to the reader (see also \cite{EZ16}).
\end{proof}

We shall consider, for simplicity, perturbations \eqref{rhoeps} where $\rho'$ is a polynomial in $Z=(z,w)$ and $\bar Z$. We may decompose the space of polynomials $\bC[Z,\bar Z]$ as follows:
\beq\label{decompP}
\bC[Z,\bar Z]=\oplus_{k=0}^\infty\oplus_{p+q=k}\mathcal H_{p,q},
\eeq
where $\mathcal H_{p,q}$ denotes the space of polynomials of bidegree $(p,q)$ in $(Z,\bar Z)$. We note that the linear partial differential operator $Q^0$ maps $P\in \mathcal H_{p,q}$ to $Q^0(P)\in \mathcal H_{p+2,q-2}$, where $\mathcal H_{a,b}$ is understood to be $\{0\}$ if $a$ or $b$ is negative. We also note that $P_{Z^2}(L_0,L_0)\in \mathcal H_{p-2,q+2}$, and we conclude that  $Q^0(P)=0$ unless both $p$ and $q$ satisfy $p,q\geq 2$. If $\rho'$ is a real-valued polynomial of degree $m$, without constant or linear terms, then we may decompose it as follows:
\beq\label{rho'decomp}
\rho'=\sum_{k=2}^m\sum_{p+q=k}\rho'_{p,q},\quad \rho_{p,q}=\overline{\rho'_{q,p}},
\eeq
where each $\rho'_{p,q}\in \mathcal H_{p,q}$. As a consequence we obtain, for $Q'$ given by \eqref{Q'},
\beq\label{Q'decomp}
Q'=\sum_{k=4}^m\sum_{l=4}^k Q'_{l,k-l},\quad Q'_{l,k-l}=Q^0(\rho'_{l-2,k-l+2}).
\eeq
Recall now that the unit sphere $M_0=S^3$ in $\bC^2$ can be identified with the unit sphere bundle $\tilde S^3$ in the universal bundle $O(-1)\to \bC\bP^1$ over the complex projective plane by blowing up the origin in $\bC^2$, and we let $\tilde Q'$ denote the function $Q'$ in \eqref{Q'} under this identification. If we let $[z,w]$ be homogeneous coordinates in $\bC\bP^1$ and $U_0$ the chart where $w\neq 0$, then we can let $\tilde z=z/w$ be a local coordinate in $U_0$ and $\zeta\mapsto \zeta(\tilde z,1)$ a local trivialization of $O(-1)|_{U_0}\to U_0$. The unit sphere $\tilde S^3$ is then given by
\beq
|\zeta|^2(|\tilde z|^2+1)=1.
\eeq
In these coordinates, the function $\tilde Q'$ can be extended off $\tilde S^3$ to $O(-1)|_{U_0}$ as a rational function $R=R(\zeta,\tilde z,\bar{\tilde z})$ of $\zeta$ with coefficients that are rational functions of $\tilde z$ and $\bar{\tilde z}$. The details can be found in \cite{EZ16}, but the upshot is that $R$ has the form
\beq\label{R1}
R=\sum_{k=4}^m\sum_{l=4}^kR_{l,k-l}=\sum_{k=4}^m\sum_{l=4}^k\frac{q_{l,k-l}(\tilde z,\bar{\tilde z})}{(1+|\tilde z|^2)^{k-l}}\,\zeta^{2l-k},
\eeq
where the coefficients $q_{l,k-l}$ are polynomials in $\tilde z$ and $\bar{\tilde z}$. If we collect terms of equal powers in $\zeta$, we may rewrite $R$ in the form
\beq\label{R2}
R=\sum_{r=8-m}^m \frac{b_r(\tilde z,\bar{\tilde z})}{(1+|\tilde z|^2)^{s_r}}\,\zeta^r
=\frac{1}{\zeta^{m-8}}\sum_{k=0}^{2m-8} \frac{b_{k+8-m}(\tilde z,\bar{\tilde z})}{(1+|\tilde z|^2)^{s_{k+8-m}}}\,\zeta^k,
\eeq
where the $s_r$ are positive integers and the $b_r$ are polynomials in $\tilde z$ and $\bar{\tilde z}$. We observe now that if, for a fixed $\tilde z\in U_0$, the rational function $R_{\tilde z}:=R(\cdot,\tilde z,\bar{\tilde z})$ does not vanish on the circle $\Gamma$ given by
\beq
|\zeta|^2=\frac{1}{1+|\tilde z|^2},
\eeq
then by the Argument Principle the winding number of $R$ along the closed curve $\Gamma$ can be computed as the difference between the number of zeros of $R_{\tilde z}$ and the number of its poles inside the disk bounded by $\Gamma$. In particular, if the $b_r$ in \eqref{R2} are zero for $r\leq 0$, then $R_{\tilde z}$ is a polynomial vanishing at $\zeta=0$, forcing the winding number to be positive, and therefore, by Proposition \ref{P:MepsQ'}, the umbilical locus of $M_\eps$ contains at least one non-trivial curve for sufficiently small $\eps$. It is shown in \cite{EZ16} that $R_{\tilde z}$ is indeed a polynomial in $\zeta$ that vanishes at $\zeta=0$ when $\rho'$ satisfies the condition $\rho'_{p,q}=0$ for $|p-q|\geq 4$. Note that in the special case where $\rho'_{p,q}=0$ for $p\neq q$, then $\rho'$ is invariant under the circular action $t\mapsto e^{it}(z,w)$, which motives the following terminology:

\begin{definition} A polynomial $P=\sum_{k=0}^m\sum_{p+q=k}P_{p,q}$ is said to be {\it almost circular} if $P_{p,q}=0$ for $|p-q|\geq 4$.
\end{definition}

If $\rho'$ is such that $\rho'_{p,q}\neq 0$ for $|p-q|=4$, then $R$, as a rational function of $\zeta$, will have a non-trivial constant term. If, furthermore, $\rho'=0$ for $|p-q|\geq 5$, then $R$ is a polynomial in $\zeta$ with a constant term, and one would need to analyse closer the dependence on $\tilde z$ of its coefficients  to determine if it has zeros inside $\Gamma$ for some choice of $\tilde z$. The reader is referred to \cite{EZ16} for more detailed formulas representing these coefficients.

We shall denote by $\AC_m$ the subspace of real-valued polynomials of degree at most $m$ that are almost circular. It is shown in \cite{EZ16} that almost circular perturbations $M_\eps$ of the unit sphere {\it generically} have a non-empty locus of umbilical points. In order to prove that $M_\eps$ has umbilical points using the arguments outlined above, we must also assert that there exists a point $\tilde z\in U_0$ such that $R$ does not vanish on the corresponding circle $\Gamma$. This will not happen for every almost circular $\rho'$, but fails only for such $\rho'$ that lie on a proper real-algebraic subvariety in $\AC_m$. This was proved in \cite{EZ16}, which resulted in the following result:

\begin{theorem}[\cite{EZ16}]\label{ACpertThm}
For $m\geq 4$, there is a real-algebraic subvariety $\Xi_m\subset \AC_m$ of dimension strictly less than that of $\AC_m$ such that if $\rho'\in \AC_m\setminus\Xi_m$, then, for sufficiently small $\eps>0$, the locus of umbilical points $\mU$ on the perturbation $M_\eps$, given by \eqref{rhoeps}, contains a curve of umbilical points. More precisely, the umbilical locus contains either a 2-surface, or a curve of stable umbilical points.
\end{theorem}

We shall conclude this paper by considering a special case of almost circular perturbations $\rho'$ that are not generic, i.e., for which $\rho'\in \Xi_m$.

\subsection{Real ellipsoids and ellipsoidal perturbations} A special family of almost circular perturbations occurs when
\beq\label{ellipsoidpert}
\rho'=\left(
	A (z^2 + \bar z^2 + 2|z|^2 )
	+ B (w^2 + \bar w^2 + 2|w|^2
	)\right),\quad A,B\geq 0,\ AB\neq 0
\eeq
In this case the CR manifolds $M_\eps$ given by \eqref{rhoeps} are real ellipsoids in case the terms $O(\eps^2)$ are not present, and ellipsoids to first order in $\eps$ in general. It is easy to see that $Q^0(\rho')=0$, and hence such perturbations $M_\eps$ will not be generic in the sense of Theorem \ref{ACpertThm}. Instead, the leading term in the asymptotic expansion of $\det A_3(\rho^\eps)$ will be the $\eps^2$-term. This term is calculated and analysed in \cite{EZ16}, and as a result the following result is obtained:

\begin{theorem}\label{ellipse}
Let $\rho'$ be given by \eqref{ellipsoidpert} and $M_\eps$ by \eqref{rhoeps}. Then, for sufficiently small $\eps>0$, the locus of umbilical points $\mU$ on the perturbation $M_\eps$ contains a curve of umbilical points. More precisely, the umbilical locus contains either a 2-surface, or a curve of stable umbilical points.
\end{theorem}

In particular, it follows that real ellipsoids close to the sphere always possess a curve of umbilical points, a fact previously proved (for all ellipsoids, not just those close to the sphere) by X. Huang and S. Ji in \cite{HuangJi07}.


\def\cprime{$'$}

\end{document}